\documentclass{scrartcl}

\usepackage{amsmath, amssymb, amsfonts}
\usepackage{amsthm}
\usepackage{mathtools}
\usepackage{graphics}
\usepackage{epstopdf}
\usepackage{natbib}
\usepackage{bbm}
\usepackage{setspace}

\usepackage{bbm}		
\usepackage{color}

\newtheorem{theorem}{Theorem}

\newtheorem{corollary}[theorem]{Corollary}

\newtheorem{proposition}[theorem]{Proposition}

\theoremstyle{definition}

\newtheorem{definition}[theorem]{Definition}
\newtheorem{remark}[theorem]{Remark}

\begin{document}
\bibliographystyle{natbib}

\title{Expectile based measures of skewness}
\author{Andreas Eberl\footnote{andreas.eberl@kit.edu}, 
Bernhard Klar\footnote{ bernhard.klar@kit.edu} \\
\small{Institute of Stochastics,} \\
\small{ Karlsruhe Institute of Technology (KIT), Germany.}
}


\date{\today }
\maketitle

\begin{abstract}
	In the literature, quite a few 
measures have been proposed for quantifying the deviation of a probability distribution from symmetry. The most popular of these skewness measures are based on the third centralized moment and on quantiles. However, there are major drawbacks in using these quantities. These include a strong emphasis on the distributional tails and a poor asymptotic behaviour for the (empirical) moment based measure as well as difficult statistical inference and strange behaviour for discrete distributions for quantile based measures.
	
Therefore, in this paper, we introduce skewness measures based on or  connected with expectiles. Since expectiles can be seen as smoothed versions of quantiles,  they preserve the advantages over the moment based measure while not exhibiting most of the disadvantages of quantile based measures.
We introduce corresponding empirical counterparts and derive  asymptotic properties. Finally, we conduct a simulation study, comparing the newly introduced measures with established ones, and evaluating the performance of the respective estimators.
\end{abstract}

{\bfseries Keywords:} Expectile, Omega ratio, skewness measure, stop-loss transform.

\section{Introduction}

Symmetry of a distribution or density function is one of the oldest and, at the same time, one of the most important concepts in probability theory and statistics. A real random variable $X$ with cumulative distribution function $F$ is symmetric about $\theta$ if $X-\theta\sim \theta -X$, or, equivalently, if $F(\theta-x)=1-F(\theta+x), x\in\mathbb{R}$.
Over time, a sizeable number of asymmetry or skewness measures has been proposed in the literature to quantify the deviation from symmetry.

The best known skewness measure is certainly the moment based  measure $\gamma_M=E[(X-EX)^3]/Var(X)^{3/2}$, which is often used synonymously with the notion of skewness itself. However, this measure has a number of disadvantages. First, it is so sensitive to the extreme tails of the distribution that it is difficult to estimate accurately in practice when the distribution is markedly skew \citep{hosking}. Second, it can not be  normalized, which makes the according skewness values less informative and comparable. Third, the asymptotic distribution of its empirical counterpart involves moments up to order 6 of the underlying distribution, implying a very slow convergence especially for heavy tailed distributions.
Finally, $\gamma_M=0$ does not characterize symmetry. On the plus side, there is its familiarity, and the fact that this measure is also reasonable for discrete distributions with finite third (or sixth) moment.

A distribution is symmetric if and only if $q_X(1-\alpha)-q_X(1/2)=q_X(1/2)-q_X(\alpha)$ for each $\alpha\in(0,1/2)$, where $q_X(\alpha)$ denotes the $\alpha$-quantile of the distribution of $X$. Hence, one can also provide a quantile based definition of skewness: A distribution is said to be right-skewed (left-skewed) in the quantile sense, if and only if
\begin{equation} \label{quantile_skewness}
q_X(1-\alpha)-q_X(1/2) \ \geq \; (\leq) \ q_X(1/2)-q_X(\alpha), \quad  \alpha\in(0,1/2),
\end{equation}
and equality does not hold for each  $\alpha\in(0,1/2)$. Corresponding scalar measures of skewness are
\begin{equation*}
b_1 = \frac{ q_X(3/4)+q_X(1/4)-2q_X(1/2) }{ q_X(3/4)-q_X(1/4) }
\end{equation*}
(see \cite{yule}, \cite{bowley}) or the general version
\begin{equation*}
b_2(\alpha) = \frac{ q_X(1-\alpha)+q_X(\alpha)-2q_X(1/2) }{ q_X(1-\alpha)-q_X(\alpha) },
\end{equation*}
where $\alpha \in (0,1/2)$, introduced by \cite{dj} (see also \cite{hinkley}, \cite{groeneveld}). As a measure of skewness, $b_1$ can be criticized for being insensitive to the distribution of $X$ any further into the tails than the quartiles \citep{hosking}. Quantiles are not unique for discrete distributions, in particular for empirical distributions.
Furthermore, statistical inference involving the asymptotic distribution of $b_1$ or $b_2$ requires the evaluation of the (unknown) density of the underlying distribution, which typically requires bandwidth selection and, hence, leads to a certain arbitrariness.  Again, rather large sample sizes are desirable to reliably represent the population being sampled.
On the other hand, $b_1$ is a robust and quite intuitive measure, and $b_2(\alpha)\equiv 0$ characterizes symmetry.

Both $b_1$ and $b_2$ satisfy three properties, introduced by \cite{zwet}, which are often seen appropriate for a skewness coefficient $\gamma$ (see \cite{oja}, \cite{groeneveld}, \cite{tajuddin}):
\begin{enumerate}
\item[S1.]
For $c>0$ and $d \in \mathbb{R}$, $\gamma(cX+d) = \gamma(X)$.
\item[S2.]
The measure $\gamma$ satisfies $\gamma(X)=-\gamma(X)$.
\item[S3.]
Let $F$ and $G$, the cdf's of $X$ and $Y$, be absolutely continuous and strictly increasing on $\{x: 0<F(x)<1\}$ and $\{x: 0<G(x)<1\}$, respectively. If $F$ is smaller than $G$ in convex transformation order (written $F\le_2 G$), i.e. $G^{-1}\left(F(x)\right)$ is convex, then $\gamma(X) \leq \gamma(Y)$.
\end{enumerate}
The convex transform order is equivalent to
\begin{align} \label{cto}
&\frac{\left( F^{-1}(w)-F^{-1}(v) \right) - \left( F^{-1}(v)-F^{-1}(u) \right)}{ F^{-1}(w)-F^{-1}(u)}  \nonumber \\
\leq
&\frac{\left( G^{-1}(w)-G^{-1}(v) \right) - \left( G^{-1}(v)-G^{-1}(u) \right)}{ G^{-1}(w)-G^{-1}(u)}
 \quad \forall \ 0<u<v<w<1.
\end{align}
Even though this characterization of the convex transformation order should be known, we were not able to find it in the literature. Thus, we give some details in Section \ref{sec-proofs}. Plugging $v=1/2, w=1-\alpha$ and $u=\alpha$ into (\ref{cto}) shows immediately that $b_2$ satisfies (S3), i.e. it preserves the convex transformation order. A different proof of this fact has been given by \cite{groeneveld}.

On the other hand, the convex transform order is the strongest of all commonly used skewness orders. Hence, also the requirement $b_{2,X}(\alpha)\le b_{2,Y}(\alpha)$ for each $\alpha \in (0,1/2)$ is a very strong one.
As \cite{ag} argue, one might be willing to forgive minor local violations, rather than to insist on a uniform domination of $b_{2,X}$ by $b_{2,Y}$, and such a proposal is made in the following.

To do so, we consider measures of skewness based on or related to expectiles. Amongst others, we propose
\begin{equation*}
 \tilde{s}_2(\alpha) = \frac{ e_X(1-\alpha)+e_X(\alpha)-2e_X(1/2) }{ e_X(1-\alpha)-e_X(\alpha) }, \quad  \alpha\in(0,1/2),
\end{equation*}
where $e_X(\alpha)$ denotes the $\alpha$-expectile of $X$, as a family of  expectile based measure of skewness. Contrary to $\gamma_M$, these measures can be normalized,  characterize symmetry, exist for any distribution with finite first moment, and have very convenient asymptotic properties. Since expectiles can be seen as smoothed quantiles, $\tilde{s}_2(\alpha)$ has a similar interpretation as $b_2(\alpha)$, but is sensitive to the whole distribution of $X$.

There are some further skewness measures with structures similar to $b_2$ and $\tilde{s}_2(\alpha)$. L-skewness, a measure based on L-moments, was introduced by \cite{hosking} and can be written as
\begin{align*}
\tau_3 = \frac{ EX_{3:3} - 2EX_{2:3} + EX_{1:3} }{ EX_{3:3} - EX_{1:3} },
\end{align*}
where $EX_{1:3}\leq EX_{2:3} \leq EX_{3:3}$ denotes the order statistics of a random sample of size 3 from the distribution of $X$. Like $\tilde{s}_2$, it exists whenever $E|X|<\infty$. As for $\gamma_M$, $\tau_3=0$ does not characterize symmetry.

\cite{cj} defined an asymmetry function as follows.
For any $a<b$, they considered the class of all rooted unimodal densities with support $(a, b)$. For any $0<p<1$, there are two points $x_L(p)$ and $x_R(p)$, one each side of the mode $m$, satisfying $f(x_L(p))=f(x_R(p))=pf(m)$. Skewness or asymmetry is measured by
\begin{equation*}
\gamma^\ast(p) = \frac{ x_R(p) -2m + x_L(p) }{ x_R(p) - x_L(p) }.
\end{equation*}
Even if the idea is interesting, it is  difficult to use this measure in practice due to the need of estimating the level points of the density as well as the modal value. Accordingly, not much is known about properties of estimators of this measure. Further, the measure is only defined for a quite restricted class of distributions. Compared with $b_2$, this is a local measure of asymmetry, whereas $\tilde{s}_2(\alpha)$ can be seen as an integral version of $b_2$.

This paper is organized as follows. In Section \ref{sec-expectile}, we recall the definitions of expectiles and some of their properties.
Section \ref{sec-exp-order} formally introduces $\tilde{s}_2$ and discusses its properties. In Section \ref{sec-or}, we introduce related skewness measures based on Omega ratios and stop-loss transforms. Empirical counterparts of the proposed measures are analyzed in Section \ref{sec-esf}.
We illustrate the behavior of the newly proposed skewness measures and their empirical counterparts for some families of distributions in Section \ref{sec-sim}. Most proofs are postponed to Section \ref{sec-proofs}.

\section{Expectiles and expectile location order} \label{sec-expectile}

Throughout the paper we assume that all mentioned random variables $X$ are non-degenerate, have a finite mean (denoted as $X \in L^1$) and are defined on a common probability space $(\Omega,\mathcal{A}, P)$ unless stated otherwise.
Recall that the expectiles $e_X(\alpha)$ of a random variable $X\in L^{2}$ have been defined by \citet{newey} as the minimizers of an asymmetric quadratic loss:%
\begin{equation}
e_X(\alpha)=\arg \min_{t\in \mathbb{R}}\left\{ E \ell_\alpha(X-t) \right\}
\text{,}  \label{exp_def_1}
\end{equation}%
where
\begin{equation*}
\ell_\alpha(x) =
\begin{cases}
\alpha x^2, & \mbox{ if } x \ge 0, \\
(1-\alpha) x^2, & \mbox{ if } x < 0,%
\end{cases}%
\end{equation*}
and $\alpha \in (0,1)$. For $X\in L^{1}$, equation (\ref{exp_def_1}) has to
be modified \citep{newey} to
\begin{equation}
e_X(\alpha) = \arg \min_{t\in \mathbb{R}} \left\{ E\left[ \ell_\alpha(X-t) -
\ell_\alpha(X) \right]\right\}.  \label{exp_def_3}
\end{equation}
The minimizer in (\ref{exp_def_1}) or (\ref{exp_def_3}) is always unique and
is identified by the first order condition%
\begin{equation}
\alpha E\left( X-e_{X}(\alpha )\right) _{+}=(1-\alpha )E\left(
X-e_{X}(\alpha )\right) _{-}\text{,}  \label{exp_def_2}
\end{equation}%
where $x_{+}=\max \{x,0\}$, $x_{-}=\max \{-x,0\}$.
This is equivalent to characterizing expectiles via an identification function, which, for any $\alpha \in (0, 1)$ is defined by
\begin{equation*}
	I_\alpha(x, y) = \alpha (y - x) \mathbbm{1}_{\{y \geq x\}} - (1 - \alpha) (x - y) \mathbbm{1}_{\{y < x\}}
\end{equation*}
for $x, y \in \mathbb{R}$. The $\alpha$-expectile of a random variable $X \in L^1$ is then the unique solution of
\begin{equation*}
	E I_\alpha(t, X) = 0, \quad t \in \mathbb{R}.
\end{equation*}
Similarly, the {\itshape empirical $\alpha$-expectile} $\hat{e}_n(\alpha)$ of a sample $X_1, ..., X_n$ is defined as solution of
\begin{equation*}
	I_\alpha(t, \hat{F}_n) = \frac{1}{n} \sum_{i=1}^{n} I_\alpha(t, X_i) = 0, \quad t \in \mathbb{R}.
\end{equation*}
As quantiles, expectiles are measures of non-central location; we collect some of their properties from \citet{newey} and \citet{bkmr} in the following proposition.

\begin{proposition}\label{exp_properties}
Let $X \in L^{1}$ with distribution function $F$ and $\alpha \in (0,1)$. Then:

\begin{enumerate}
\item[a)]
$e_{X+h}(\alpha )=e_{X}(\alpha )+h$, for each $h\in \mathbb{R}$,
\item[b)]
$e_{\lambda X}(\alpha )=\lambda e_{X}(\alpha )$, for each $\lambda >0$,
\item[c)]
$e_{X}(\alpha )$ is strictly increasing with respect to $\alpha$,
\item[d)]
$e_{X}(\alpha )$ is continuous with respect to $\alpha$,
\item[e)]
$e_{-X}(\alpha )=-e_{X}(1-\alpha )$,
\item[f)]
for continuous cdf $F$, $e_{X}$ has derivative
\begin{equation*}
e_{X}^{\prime}(\alpha )
= \frac{E\left\vert X-e_{X}(\alpha)\right\vert }{(1-\alpha )F(e_{X}(\alpha ))+\alpha \left(1-F(e_{X}(\alpha))\right)},
\end{equation*}
\item[g)]
$X\leq Y$ a.s. $\Rightarrow e_{X}(\alpha )\leq e_{Y}(\alpha )$, for each $\alpha \in (0,1)$.
\end{enumerate}
\end{proposition}

Clearly, expectiles depend only on the distribution of the random variable
$X$; they can be seen as statistical functionals defined on the set of distribution functions with finite mean on $\mathbb{R}$.

Quantiles and expectiles are closely connected as measures of non-central location. \cite{bell-et-al} introduced the expectile (location) order between two random variables: Two random variables $X,Y \in L^1$ are ordered in expectile order (written $X\leq _{e}Y$)
if $e_X(\alpha)\leq e_Y(\alpha) \mbox{ for all } \alpha \in (0,1)$.

It is well known that the usual stochastic order $\leq_{st}$ is equivalent to the pointwise ordering of the quantiles. In view of this, the preceding definition seems quite natural, since quantiles are just replaced by expectiles.

The next theorem shows that the usual stochastic order implies the expectile order, i.e. the ordering of the quantiles implies the ordering of the expectiles.

\begin{theorem} \label{theorem:exp-loc-order}
Let  $X,Y \in L^1$. Then, $X\leq _{st}Y$ implies $X\leq _{e}Y$.
\end{theorem}

This has been proved by \cite{bellini} by an order-theoretic comparative static approach. It follows also directly from Theorem \ref{exp_properties} g) and the well-known fact that $X\leq _{st}Y$ is equivalent to the existence of random variables $\hat{X}$ and $\hat{Y}$, defined on the same probability space, with $\hat{X} =_{st} X, \hat{Y} =_{st} Y,$ and $P(\hat{X} \leq \hat{Y})=1$
(here $=_{st}$ denotes equality in law).

\section{An expectile based ordering with respect to skewness} \label{sec-exp-order}

The distribution of a random variable $X$ is symmetric around $\theta$ if $X-\theta \sim \theta-X$.
Since a distribution is uniquely determined by the expectile function, and since the mean $\mu=e_X(1/2)$ of $X$ coincides with the center of symmetry for a symmetric distribution, it follows that a distribution is symmetric if and only if
$e_{X-\mu}(\alpha)=e_{\mu-X}(\alpha)$ for each $\alpha\in(0,1)$.
Using properties a) and e) in Proposition \ref{exp_properties}, this is equivalent to
$$
e_X(1-\alpha)-\mu=\mu-e_X(\alpha)  \quad \mbox{for each} \quad \alpha\in(0,1/2).
$$
In analogy to (\ref{quantile_skewness}), this leads to the following definition of expectile based skewness.

\begin{definition} \label{expectile-skew}
A distribution is called right-skewed (left-skewed) in the expectile sense, if and only if
\begin{equation} \label{skew-cond1}
e_X(1-\alpha)-\mu \ \geq \; (\leq) \ \mu-e_X(\alpha), \quad  \alpha\in(0,1/2),
\end{equation}
and equality does not hold for each  $\alpha\in(0,1/2)$.
\end{definition}

Corresponding scalar measures of skewness are
\begin{equation} \label{skew-coef1}
\tilde{s}_1 = \frac{ e_X(3/4)+e_X(1/4)-2\mu }{e_X(3/4)-e_X(1/4)}
\end{equation}
or the general version
\begin{equation} \label{skew-coef2}
\tilde{s}_2(\alpha) = \frac{ e_X(1-\alpha)+e_X(\alpha)-2\mu }{ e_X(1-\alpha)-e_X(\alpha) }, \quad  \alpha\in(0,1/2).
\end{equation}
The numerator of (\ref{skew-coef2}) is the difference of two positive numbers $e_X(1-\alpha)-\mu$ and $\mu-e_X(\alpha)$, while the denominator is the sum of these numbers. Hence, $-1\leq \tilde{s}_2(\alpha)\leq 1$.

Note that the following decomposition holds for the expectile as a measure of non-central tendency:
\begin{eqnarray} \label{exp_decomp}
e_X(1-\alpha) &=& e_X(1/2) + 1/2\left\{e_X(1-\alpha)-e_X(\alpha)\right\} \nonumber \\
 && + 1/2\left\{e_X(1-\alpha)+e_X(\alpha)-2e_X(1/2) \right\}.
\end{eqnarray}
This is the counterpart to the decomposition given in \cite{bk} for the quantile. The first term in (\ref{exp_decomp}), the mean, is a measure of central location; the second term, half of the expectile distance $e_X(1-\alpha)-e_X(\alpha)$, is a measure of variability for $\alpha<1/2$. Finally, the third term, which is essentially the numerator in (\ref{skew-coef2}), is zero for symmetric distributions and, hence, quantifies the deviations from symmetry.

\medskip \noindent
The actual range of $\tilde{s}_2(\alpha)$ depends on $\alpha$ and can be considerably smaller than the interval $[-1, 1]$. This is specified in the following result whose proof can be found in Section \ref{sec-proofs}.

\begin{proposition}	\label{thm:exlSkewStd}
Let $\alpha \in (0, 1/2)$. Then, $-1 + 2\alpha < \tilde{s}_2(\alpha) < 1 - 2\alpha$ for all random variables, for which $\tilde{s}_2(\alpha)$ is defined, and both bounds cannot be improved.
\end{proposition}

Based on this result we redefine our expectile based skewness measure as
\begin{equation*}
	s_2(\alpha) = \frac{1}{1 - 2\alpha} \tilde{s}_2(\alpha), \quad \alpha \in (0, 1/2).
\end{equation*}
Then, $-1 < s_2(\alpha) < 1$, and both inequalities are sharp for any $\alpha \in (0, 1/2)$. Further, we define $s_1 = s_2(1/4) = 2 \tilde{s_1}$.

\medskip \noindent

Clearly, for the comparison of the skewness of two random variables, it doesn't matter if we use $\tilde{s}_2$ or $s_2$. Hence, we say that $Y$ is more skewed to the right than $X$ in the expectile sense if
$\tilde{s}_{2,X}(\alpha) \leq \tilde{s}_{2,Y}(\alpha)$ for each $\alpha \in (0,1/2)$.
Using the properties of expectiles, it is clear that $\tilde{s}_2(\alpha)$ satisfies the skewness properties S1 and S2. The validity of S3 is an open question.

An even more general definition of an expectile based skewness order would be the analogue to display (\ref{cto}); however, this seems to be rather unmanageable in applications.

\subsection{The limiting case $\alpha\to 1/2$}

In this section, we examine the behaviour of $s_2(\alpha)$ if $\alpha$ approaches its upper bound $1/2$. For this, we assume that the cdf $F$ of $X$ is differentiable with density $f$; then, $e_X(\alpha)$ is twice differentiable by Theorem \ref{exp_properties}f). Note that we can rewrite $s_2$ as a ratio of first- and second-order difference quotients
\begin{equation} \label{diff-quot}
s_2(\alpha) = 1/2 \frac{(e_X(1/2 + h) - 2 e_X(1/2) + e_X(1/2 - h))/h^2}{(e_X(1/2 + h) - e_X(1/2 - h))/h},
\end{equation}
where $h = 1/2 - \alpha$. Splitting the central difference in the denominator into a forward and a backward difference, and taking the left limit yields
\begin{equation*}
\lim\limits_{\alpha \rightarrow 1/2-} s_2(\alpha) = \frac{e_X''(1/2)}{4 e_X'(1/2)}.
\end{equation*}
By Proposition \ref{exp_properties}f), $e_X'(1/2) = 2 E|X - \mu| = 2 \delta_X$, say. For the calculation of $e_X''(1/2)$, we denote numerator and denominator of $e_X'(\alpha)$ in Theorem \ref{exp_properties}f) by $u(\alpha)$ and $v(\alpha)$, respectively.
Then, $\lim_{\alpha \rightarrow 1/2} u(\alpha) = \delta_X$ and $\lim_{\alpha \rightarrow 1/2} v(\alpha) = 1/2$ as well as
\begin{align*}
u'(\alpha) &= e_X'(\alpha) (2 F(e_X(\alpha)) - 1) \rightarrow e_X'(1/2) (2 F(\mu) - 1),\\
v'(\alpha) &= (1 - 2 F(e_X(\alpha))) + (1 - 2 \alpha) f(e_X(\alpha)) e_X'(\alpha) \rightarrow 1 - 2 F(\mu),
\end{align*}
for $\alpha \rightarrow 1/2$. By combining these results, it follows
\begin{equation*}
e_X''(1/2) = \lim\limits_{\alpha \rightarrow 1/2} \frac{u'(\alpha) v(\alpha) - u(\alpha) v'(\alpha)}{(v(\alpha))^2} = 8 \delta_X (2 F(\mu) - 1),
\end{equation*}
which overall yields
\begin{equation} \label{s3}
s_3 = \lim_{\alpha \rightarrow 1/2-} s_2(\alpha) = 2 F(\mu) - 1.
\end{equation}
Apparently, $s_3$ can also be used as a measure of skewness; in fact, it has already been introduced as such by \cite{tajuddin}. Besides, the quantity $F(\mu)$ is the theoretical counterpart of the test statistic of the sign test for symmetry with estimated center \citep{gastwirth}.
The measure $s_3$ exploits the idea that the difference between mean $\mu$ and median $q_{1/2}$ indicates the skewness of the underlying distribution, which is also prevalent in other popular skewness measures like $(\mu - q_{1/2})/\sigma$ or $(\mu - q_{1/2})/E|X - q_{1/2}|$. Since a substitution of the mean by the median in $s_3$ always results in the value $0$ for continuous distributions, a positive value of $\mu - q_{1/2}$ yields a positive value of $s_3$ and thus right-skewness and vice versa.

It is easy to see that $s_3$ satisfies the skewness properties S1 generally and S2 under the assumption $P(X = \mu) = 0$. For continuous distributions, the crucial property S3 follows from Jensen's inequality.

Besides its simplicity, an argument for the use of $s_3$ is that $s_2'(\alpha)$ converges to zero as $\alpha$ tends to $1/2$, so $s_2(\alpha)$ flattens out towards $s_3$. This means that, at least for values of $\alpha$ close to $1/2$, $s_3$ is close to and thereby representative for a range of values of $s_2(\alpha)$ without the need of a specific choice of the parameter $\alpha$. This result on the gradient of $s_2$ can be proved under the assumption of $e_X \in C^4((0, 1))$ (which is equivalent to the assumption that the density of $X$ is twice differentiable) as follows.

First, we differentiate $s_2$ with respect to $\alpha$, which yields
\begin{align*}
	s_2'(\alpha) =& \ 2 \frac{e_X'(\alpha) (e_X(1 - \alpha) - e_X(1/2)) - e_X'(1 - \alpha) (e_X(1/2) - e_X(\alpha))}{(1 - 2 \alpha) (e_X(1 - \alpha) - e_X(\alpha))^2}\\
	&+ 2 \frac{e_X(1 - \alpha) - 2 e_X(1/2) + e_X(\alpha)}{(1 - 2 \alpha)^2 (e_X(1 - \alpha) - e_X(\alpha))}
\end{align*}
for $\alpha \in (0, 1/2)$. Using the notation $h = 1/2 - \alpha$, this can once again be rewritten as a composition of difference quotients. Now, using Taylor expansions for each of them such that the remainders are of order $O(h^3)$ in the numerators as well as in the denominators yields after some computations
\begin{align*}
s_2'(\alpha) =& \ 1/4 \frac{- e_X'(1/2) e_X''(1/2) h + O(h^3)}{(e_X'(1/2))^2 h^2 + O(h^3)}
+ 1/4 \frac{e_X''(1/2) + 1/12 e_X^{(4)}(1/2) h^2 + O(h^3)}{e_X'(1/2) h + O(h^3)}\\
=& \ 1/4 \frac{O(h^3)}{e_X'(1/2) h^2 + O(h^3)},
\end{align*}
where we used for the second equality that $e_X'(1/2) > 0$ by Proposition \ref{exp_properties}c). Then, taking the limit $h\to 0$ yields the asserted result $\lim_{\alpha \to 1/2} s_2'(\alpha) = \lim_{h \to 0} s_2'(\alpha) = 0$.


\section{Relation to Omega ratios and stop-loss transforms} \label{sec-or}

In this section, we give conditions based on Omega ratios and stop-loss transforms which are equivalent to Definition \ref{expectile-skew}.

Expectiles are related to the Omega ratio, which
has been introduced in the financial literature by \citet{keating} as
\begin{equation}
\Omega _{X}(t)=\frac{E\left( X-t\right) _{+}}{E\left( X-t\right) _{-}}\text{.%
}  \label{omega}
\end{equation}
Then, equation (\ref{exp_def_2}) can be written as
\begin{equation}
\Omega _{X}(e_X(\alpha))=\frac{1-\alpha }{\alpha }  \label{exp_omega}
\end{equation}%
\citep{remillard}, which gives the following one-to-one relation between expectiles and Omega
ratios:
\begin{equation*}
e_X(\alpha) =\Omega _{X}^{-1}\left( \frac{1-\alpha }{\alpha } \right) \text{,%
} \quad \Omega _{X}(t) = \frac{1-e_X^{-1}(t) }{e_X^{-1}(t) } .
\label{exp_omega_2}
\end{equation*}
Arguing as before Definition \ref{expectile-skew}, condition (\ref{skew-cond1}) for right-skewness is equivalent to
\begin{equation*}
e_{X-\mu}(\alpha) \ \geq \ e_{-(X-\mu)}(\alpha), \quad  \alpha\in(1/2,1).
\end{equation*}
For $\alpha\in(1/2,1)$, put $\beta =(1-\alpha)/\alpha\in(0,1)$. From equation (\ref{exp_omega}), the condition $e_{-(X-\mu)}(\alpha) \leq e_{X-\mu}(\alpha)$ is equivalent to
\begin{equation} \label{exp_comp}
\Omega _{-(X-\mu)}(x)=\beta \text{, }\Omega _{X-\mu}(y)=\beta \ \Rightarrow \  x\leq y.
\end{equation}%
Since $\Omega _{-(X-\mu)}(0)=\Omega _{X-\mu}(0)=1$, since $\Omega _{-(X-\mu)}$ and $\Omega _{X-\mu}$ are strictly decreasing and since (\ref%
{exp_comp}) holds for each $\beta \in (0,1)$, it follows that
\begin{equation} \label{skew-cond2}
\Omega _{-(X-\mu)}(t)\leq \Omega _{X-\mu}(t)\text{ for all } t\in (0,\infty).
\end{equation}
Using $\Omega _{X-\mu}(t)=\Omega _{X}(\mu+t)$ and $\Omega_{-(X-\mu)}(t)=1/\Omega_{X}(\mu-t)$, we finally obtain that (\ref{skew-cond2}), and, hence, condition (\ref{skew-cond1}) for right-skewness, is equivalent to
\begin{equation} \label{skew-cond3}
\Omega_{X}(\mu+t) \cdot \Omega_{X}(\mu-t) \geq 1 \quad \mbox{for all} \; t>0.
\end{equation}
Further note that a distribution is symmetric if and only if equality holds in (\ref{skew-cond3}) for every $t>0$.

\medskip
The Omega ratio, in turn, is closely related to the  {\em stop-loss transform}
$\pi_X(t) = E(X-t)_+$, which is well known in the actuarial literature (see e.g. \cite{mueller}), since it describes the expected cost of a stop-loss insurance contract with deductible $t$ for a risk $X$. From
\begin{equation*}
E(X-t)_- = t - EX + E(X-t)_+
\end{equation*}
we immediately get
\begin{equation} \label{or-slt}
\Omega_X(t) = \frac{\pi_X(t)}{t-EX + \pi_X(t)}.
\end{equation}
Plugging (\ref{or-slt}) into condition (\ref{skew-cond3}), we obtain the following result:

\begin{theorem} \label{slt-symmetry}
\begin{enumerate}
\item[a)]
The distribution of a random variable $X$ with cdf $F$ and finite mean $\mu=EX$ is symmetric around $\mu$ if and only if
\begin{equation*}
S_X(t) = \frac{1}{t} \left\{ \pi_X(\mu+t) - \pi_X(\mu-t) \right\} + 1\ = \ 0  \quad \text{for each } t>0.
\end{equation*}
\item[b)]
A distribution is right-skewed (left-skewed) in the sense of Definition \ref{expectile-skew} if and only if
\begin{equation*}
S_X(t) \ \geq (\leq) \ 0  \quad \text{for each } t>0.
\end{equation*}
\end{enumerate}
\end{theorem}

The following proposition collects some properties of the skewness function $S_X$.

\begin{proposition} \label{skewness-function}
Let $X$ be random variable with cdf $F$ and finite mean $\mu$. Then:
\begin{enumerate}
  \item[a)] $\lim_{t\to\infty} S_X(t) = 0$.
  \item[b)]
$S_X(t) = \frac{1}{t} \int_{\mu-t}^{\mu + t} F_X(z) dz - 1$.
  \item[c)]
  $ -1 \ \leq \  S_X(t)  \ \leq \ 1.$
\end{enumerate}
\end{proposition}

\begin{proof}
Since $\lim_{t\to \infty} \pi_X(\mu+t)=0$, and since the monotone convergence theorem implies
\begin{eqnarray*}
\lim_{t\to \infty} \left\{ \pi_X(\mu-t)-t\right\}
&=& \lim_{t\to \infty} E\left[\max\{X-\mu,-t\}\right] = 0,
\end{eqnarray*}
we obtain $\lim_{t\to\infty} S_X(t) = 0$. Further,
\begin{eqnarray*}
S_X(t) &=& \frac{1}{t} \left\{ \int_{\mu+t}^{\infty} \bar{F}_X(z) dz - \int_{\mu-t}^{\infty} \bar{F}_X(z) dz \right\} + 1 \\
&=& 1 - \frac{1}{t} \int_{\mu-t}^{\mu + t} \bar{F}_X(z) dz
\ = \ \frac{1}{t} \int_{\mu-t}^{\mu + t} F_X(z) dz - 1,
\end{eqnarray*}
where $\bar{F}_X(z)=1-F_X(z)$ denotes the survivor function. Part c) follows directly from b).
\end{proof}

\begin{figure}
\centering{
\includegraphics[scale=0.5]{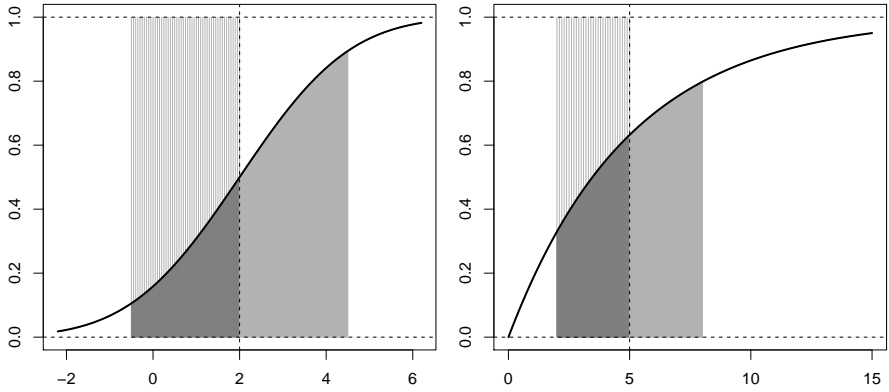}
\caption{\label{fig2} Area below $F_X(z)$ for $z\in [\mu-t,\mu+t]$ for a (symmetric) normal distribution $N(2,4)$ (left panel) and a right-skewed exponential distribution with mean 5 (right panel).}
}
\end{figure}

Figure \ref{fig2} illustrates the area below $F_X(z)$ for $z\in [\mu-t,\mu+t]$ for a (symmetric) normal distribution $N(2,4)$ (left panel, $t=2.5$) and a right-skewed exponential distribution with mean 5 (right panel, $t=3$). For the normal distribution, the gray areas below $F_X$ sum up to $t=2.5$, whereas the sum is larger than $t=3$ in case of the exponential distribution.

\begin{remark}
	\begin{enumerate}
		\item[(i)]
		The representation of $S_X$ in Proposition \ref{skewness-function}c) bears some similarity to skewness functionals defined in \cite{ag}. In particular, they proposed
		\begin{align*}
			\lambda_X(u) &= \int_{0}^{u} F^{-1}(1/2+v)+F^{-1}(1/2-v) dv, \ 0<u<1/2,
		\end{align*}
		as skewness function. Note, however, that this is a skewness measure with respect to the median, whereas $S_X$ is a measure with respect to the mean (cf. \cite{macgillivray}).
		
		\item[(ii)]
		Note that $S_X(t)=2 \int_{\mathbb{R}} F_X(z) dH(z) - 1$,
		where $H$ is the cdf of the uniform distribution on $(\mu-t,\mu+t)$.
		Replacing $H$ by the Dirac measure in $\mu$ results in $s_3$ in (\ref{s3}).
		Another reasonable choice for $H$ would be any cdf with unimodal density, that is symmetric around $\mu$, for example a normal distribution with mean $\mu$.
	\end{enumerate}
\end{remark}

The skewness function $S_X(t)$ is location invariant, but not scale invariant.
This is not an issue if one analyzes the skewness of a single distribution.
However, scale invariance is essential for a meaningful comparison between several distributions. As a scale invariant modification, we propose
\begin{equation} \label{tildeS}
\tilde{S}_X(t) \ = \ S_X(t\delta_X),
\end{equation}
where $\delta_X=E|X-EX|$ denotes the mean absolute deviation from the mean (MAD). This dispersion measure is strongly related to the stop-loss transform, since $\pi_X(EX)=\delta_X/2$ is just the absolute semideviation.
In principle, one could use any other dispersion measure $\sigma_X$ satisfying $\sigma_{cX}=c\sigma_X$ for $c>0$ instead of  $\delta_X$, but the latter is particularly suitable for our purpose.

\medskip
From now on we assume that the cdf's are absolutely continuous and strictly increasing (i.e. $F$ is strictly increasing on $\{  x:0<F(x)<1 \}$).
\citet{oja} showed that $F$ and $G$ are strongly skewness comparable (i.e. $F\leq_2 G$ or $G\leq_2 F$) if and only if $F(x)$ and $G(ax+b)$ cross each other at most twice for all $a>0,b\in \mathbb{R}$. He then defined two weakenings of $\leq_2$ in case of finite expectations $\mu_F$ and $\mu_G$ and  finite variances $\sigma_F$ and $\sigma_G$ as follows.
\begin{itemize}
\item
$F\leq_2^* G$ if the standardized distribution functions $F(\sigma_F x+\mu_F)$ and $G(\sigma_G x+\mu_G)$ cross each other exactly once on each side of $x=0$, with $F(\mu_F)\leq G(\mu_G)$.
\item
$F\leq_2^{**} G$ if there exist $a>0,b\in \mathbb{R}$ such that $F(x)$ and $G(a x+b)$ cross each other exactly twice with $F(x)-G(a x+b)$ changing sign from positive to negative to positive.
\end{itemize}
The following implications hold true \citep{oja}:
\begin{equation*}
 F\leq_2 G \quad \Rightarrow \quad F\leq_2^{*} G \quad \Rightarrow \quad F\leq_2^{**} G.
\end{equation*}
Similarly like $F\leq_2^* G$ (see also Def. 2.1 in \cite{macgillivray}), we now define skewness with respect to mean and MAD:

\begin{definition} \label{def-mu-skew}
$G$ is more skew with respect to mean and MAD than $F$ ($F <_\mu^{\delta} G$), if $F(\delta_F x+\mu_F)$ and $G(\delta_G x+\mu_G)$ cross each other exactly once on each side of $x=0$, with $F(\mu_F)\leq G(\mu_G)$.

More generally, $F\leq_\mu^{\delta} G$ if $F <_\mu^{\delta} G$ or $F(\delta_F \cdot +\mu_F)$ and $G(\delta_G \cdot+\mu_G)$ are identical.
\end{definition}

The next theorem shows that this new skewness order is weaker than the strong skewness order $\leq_2$; on the other hand, it is stronger than the skewness order implied by $\tilde{S}_X$ given in (\ref{tildeS}).

\begin{theorem} \label{theorem-A4}
Let $X \sim F$ and $Y \sim G$ with finite expectations $\mu_F=EX$ and $\mu_G=EY$. Then:
\begin{enumerate}
\item[a)]
$ F\leq_2 G \quad \Rightarrow \quad F\leq_\mu^{\delta} G \quad \Rightarrow \quad F\leq_2^{**} G.$
\item[b)]
$ F\leq_\mu^{\delta} G \quad \Rightarrow \quad \tilde{S}_X(t) \leq \tilde{S}_Y(t) \ \forall \, t>0.$ \\
In particular, the skewness measure $\tilde{S}_X(t)$ satisfies skewness property S3 for any $t>0$.
\end{enumerate}
\end{theorem}

\begin{remark}
\begin{enumerate}
\item
The proof of Theorem \ref{theorem-A4} is postponed to Section \ref{sec-proofs}; it shows that it is reasonable to require exactly two crossings in Definition \ref{def-mu-skew}. Exactly one crossing can occur only in specific situations where the standardized distribution functions are identical for all values smaller (larger) than zero; in this case, there is no reasonable comparison between the skewness of the two cdf's.
\item
 From Theorem \ref{theorem-A4} and the strong connection between $\tilde{S}_X$ and the skewness measures $s_2(\alpha)$ in (\ref{skew-coef2}) we conjecture that $s_2$ also satisfies property S3. This is reinforced by the validity of S3 for the limiting measure $s_3(\alpha)$, and by numerical computations for specific examples, see Section \ref{sec-sim}.
 \end{enumerate}
\end{remark}

\section{Empirical expectile skewness} \label{sec-esf}

By replacing the theoretical expectiles by empirical ones we obtain the plug-in estimator
\begin{align*}
	\hat{s}_{2, n}(\alpha) = \frac{1}{1 - 2 \alpha} \frac{\hat{e}_n(1 - \alpha) + \hat{e}_n(\alpha) - 2 \bar{x}}{\hat{e}_n(1 - \alpha) - \hat{e}_n(\alpha)}, \quad \alpha \in (0, 1/2)
\end{align*}
and $\hat{s}_{1, n} = \hat{s}_{2, n}(1/4)$. Utilizing the asymptotic normality of a finite number of expectiles (see, e.g., \cite{holzmann}) and the fact that $s_2(\alpha)$ is a differentiable function of expectiles, the delta method yields the following theorem. Preliminarily, we define
\begin{equation*}
	\eta(\tau_1, \tau_2) = E[I_{\tau_1}(e_X(\tau_1), X) I_{\tau_2}(e_X(\tau_2), X)]
\end{equation*}
for $\tau_1, \tau_2 \in (0, 1)$ and
\begin{equation*}
	A(\tau) = (2 \mathbbm{1}_{\{\tau < 1/2\}} - 1) \frac{e_X(1 - \tau) - \mu}{\tau + F(e_X(\tau)) (1 - 2 \tau)}
\end{equation*}
for $\tau \in \{\alpha, 1 - \alpha\}$.

\begin{theorem}
	Let $\alpha \in (0, 1/2)$ and let $X, X_1, X_2, ...$ be iid with cdf $F$ such that $EX^2 < \infty$. Furthermore, let $F$ not have a point mass in $e_X(p)$ for $p \in \{\alpha, 1/2, 1 - \alpha\}$. Then
\begin{equation*}
\sqrt{n} (\hat{s}_{2, n}(\alpha) - s_2(\alpha)) \stackrel{\cal{D}}{\longrightarrow} N(0, \sigma_\alpha^2),
\end{equation*}
where
\begin{align*}
\sigma_\alpha^2 = & \ \frac{4}{(1 - 2 \alpha)^2} \bigg[  & \\
& \frac{4 \eta(1/2, 1/2)}{(e_X(1 - \alpha) - e_X(\alpha))^2} - \frac{4 [A(\alpha) \eta(\alpha, 1/2) + A(1 - \alpha) \eta(1/2, 1 - \alpha)]}{(e_X(1 - \alpha) - e_X(\alpha))^3} \\
& + \frac{(A(\alpha))^2 \eta(\alpha, \alpha) + A(\alpha) A(1 - \alpha) \eta(\alpha, 1 - \alpha) + (A(1 - \alpha))^2 \eta(1 - \alpha, 1 - \alpha)}{(e_X(1 - \alpha) - e_X(\alpha))^4} \bigg].
\end{align*}
\end{theorem}

It can also be shown that $\hat{s}_{2, n}(\alpha)$ is a strongly consistent estimator for $s_2(\alpha)$ for all $\alpha \in (0, 1/2)$. To see this, we consider Theorem 2 in \citet{holzmann}, which implies strong consistency of $\hat{e}_n(\tau)$ for $e_X(\tau)$ for all $\tau \in (0, 1)$ under weak assumptions. Using Slutzky's Theorem, the almost sure convergence then also holds for any continuous function of any finite number of expectiles, yielding the following corollary.

\begin{corollary}
	Let $\alpha \in (0, 1/2)$ and let $X, X_1, X_2, ...$ be iid with cdf $F$ such that $E|X| < \infty$. Then, $\hat{s}_{2, n}(\alpha)$ is strongly consistent for $s_2(\alpha)$, i.e. $\hat{s}_{2, n}(\alpha) \stackrel{a.s.}{\longrightarrow} s_2(\alpha)$.
\end{corollary}

In order to obtain the plug-in estimator $\widehat{\sigma}_\alpha^2$ of $\sigma_\alpha^2$, the expectiles $e_X(\tau)$ are replaced by the empirical expectiles $\hat{e}_n(\tau)$. Moreover, $\eta(\tau_1, \tau_2)$ and $A(\tau)$ are estimated by
\begin{equation*}
	\hat{\eta}_n(\tau_1, \tau_2) = \frac{1}{n} \sum_{i = 1}^{n} I_{\tau_1}(\hat{e}_n(\tau_1), X_i) I_{\tau_2}(\hat{e}_n(\tau_2), X_i)
\end{equation*}
and
\begin{equation*}
	\hat{A}_n(\tau) = (2 \mathbbm{1}_{\{\tau < 1/2\}} - 1) \frac{\hat{e}_n(1 - \tau) - \bar{X}}{\tau + \hat{F}_n(\hat{e}_n(\tau)) (1 - 2 \tau)},
\end{equation*}
where $\hat{F}_n$ denotes the empirical cdf. It is then easy to see that $\widehat{\sigma}_\alpha^2$ is a composition of consistent estimators, hence $\widehat{\sigma}_\alpha^2$ itself is a consistent estimator of $\sigma_\alpha^2$. Consequently, an asymptotic confidence interval for $s_2(\alpha)$ with confidence level $1 - p$ is given by
\begin{equation}  \label{eqn:limits-exl-skewness}
	\hat{s}_{2, n}(\alpha) - \frac{\widehat{\sigma}_\alpha}{\sqrt{n}} z_{1 - p/2} \leq s_2(\alpha) \leq \hat{s}_{2, n}(\alpha) + \frac{\widehat{\sigma}_\alpha}{\sqrt{n}} z_{1 - p/2},
\end{equation}
where $z_q$ denotes the $q$-quantile of the standard normal distribution. The left panel in Figure \ref{fig:SE-limits} shows a plot of $\hat{s}_{2, n}(\alpha)$ with (pointwise) confidence limits for a sample of size $n=50$ from an exponential distribution with rate 1.

\begin{figure}
	\centering{
		\includegraphics[scale=0.44]{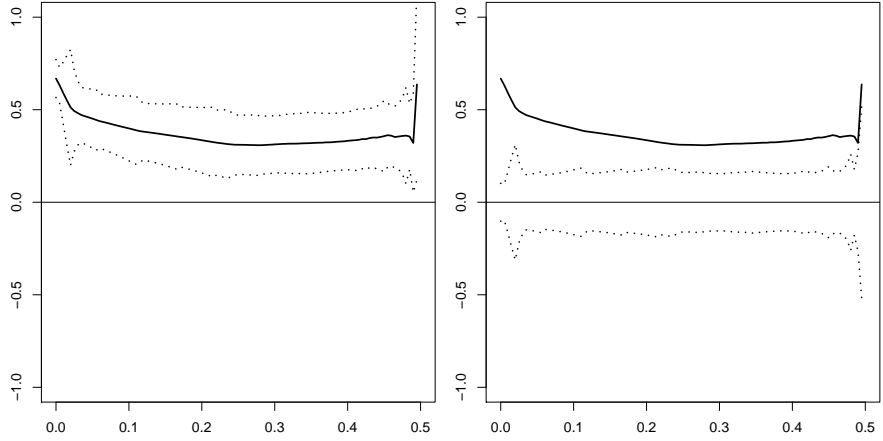}
		\caption{\label{fig:SE-limits} Left panel: Plot of $\hat{s}_{2, n}(\alpha)$ with (pointwise) 95\%-confidence limits defined in (\ref{eqn:limits-exl-skewness}) as dotted line.
			Right panel: Plot of $\hat{s}_{2, n}(\alpha)$ with limits under the assumption of symmetry defined in (\ref{eqn:SE-limits-H0}) as dotted line.}
	}
\end{figure}

Under the hypothesis of symmetry, $s_2(\alpha) \equiv 0$.
Therefore,
\begin{equation}  \label{eqn:SE-limits-H0}
\lim_{n\to\infty} P\left( - \frac{ \widehat{\sigma}_\alpha }{ \sqrt{n} } \cdot z_{1-p/2}
\leq \hat{s}_{2, n}(\alpha) \leq \frac{ \widehat{\sigma}_\alpha }{ \sqrt{n} } \cdot z_{1-p/2} \right) = 1-\alpha,
\end{equation}
which define confidence limits under the hypothesis of symmetry. The right panel in Figure \ref{fig:SE-limits} shows a plot of $\hat{s}_{2, n}(\alpha)$ together with the limits given in (\ref{eqn:SE-limits-H0}) for the same data set as in the left panel.

\subsection{The empirical skewness function}

Again, we use the plug-in estimator for $S(t)$, which is given by
\begin{align*}
S_n(t) &= \frac{1}{nt} \sum_{i=1}^n \left\{ (X_i-\bar{X}-t)_+ - (X_i-\bar{X}+t)_+ \right\} + 1.
\end{align*}
We have the following result, whose proof is again postponed to Section \ref{sec-proofs}.

\begin{theorem} \label{asym-skew-function}
Let $X_1, X_2,\ldots$ be iid with continuous cdf $F$ and $EX^2 < \infty$. Then,
  \begin{align*}
\sqrt{n} \left( S_n(t)-S(t) \right) & \stackrel{\cal{D}}{\longrightarrow} N(0, \sigma_t^2),
\end{align*}
where
\begin{align*}
\sigma_t^2 &= \frac{1}{t^2} \cdot Var\left( (X_1-\mu-t)_+ - (X_1-\mu+t)_+
 + (X_1-\mu) \left(F(\mu+t)-F(\mu-t)\right) \right).
\end{align*}
\end{theorem}

\begin{figure}
\centering{
\includegraphics[scale=0.5]{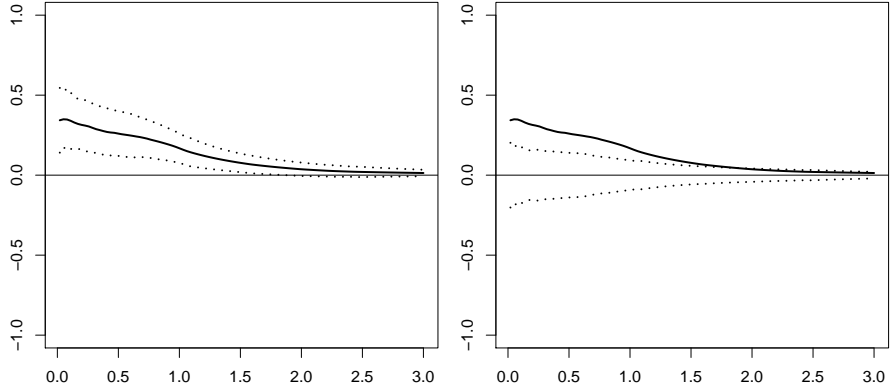}
\caption{\label{fig1} Left panel: Plot of $S_n(t)$ with (pointwise) 95\%-confidence limits defined in (\ref{limits-skewness-function}) as dotted line.
Right panel: Plot of $S_n(t)$ with limits under the assumption of symmetry defined in (\ref{limits-H0}) as dotted line.}
}
\end{figure}

The plug-in estimator for $\sigma_t^2$ is
\begin{eqnarray*}
\widehat{\sigma}_t^2
&= & \frac{1}{nt^2} \sum_{i=1}^n \left( (X_i-\bar{X}-t)_+ - (X_i-\bar{X}+t)_+
  + (X_i-\bar{X}) \, \hat{p}_t \, \right)^2  \\
&& - \left\{ \frac{1}{nt} \sum_{i=1}^n \left( (X_i-\bar{X}-t)_+ - (X_i-\bar{X}+t)_+ \right) \right\}^2 \, ,
\end{eqnarray*}
where
$\ \hat{p}_t = 1/n \sum_{i=1}^{n} \mathbbm{1} \{ \mu-t<X_i\leq \mu+t \}.$
Analogous to the expectile skewness $s_2(\alpha)$, an asymptotic confidence interval for $S(t)$ with confidence level $1-p$ is given by
\begin{equation}  \label{limits-skewness-function}
S_n(t) - \frac{ \widehat{\sigma}_t }{ \sqrt{n} } \cdot z_{1-p/2} \leq S(t) \leq S_n(t) + \frac{ \widehat{\sigma}_t }{ \sqrt{n} } \cdot z_{1-p/2}.
\end{equation}
Exemplary, (pointwise) confidence limits for a sample of size 50 from an exponential distribution with rate $1$ are given in the left panel of Figure \ref{fig1}.

As before, the hypothesis of symmetry yields $S(t) \equiv 0$ and, thereby,
\begin{equation}  \label{limits-H0}
\lim_{n\to\infty} P\left( - \frac{ \widehat{\sigma}_t }{ \sqrt{n} } \cdot z_{1-\alpha/2}
\leq S_n(t) \leq \frac{ \widehat{\sigma}_t }{ \sqrt{n} } \cdot z_{1-\alpha/2} \right) = 1-\alpha,
\end{equation}
the respective confidence limits are given in the right panel of Figure \ref{fig1}.

\section{Expectile and quantile skewness for some families of distributions} \label{sec-sim}

\subsection{Comparison of theoretical values}
In this section we examine how the expectile skewness $s_2(\alpha)$ behaves for specific families of continuous distributions, in particular for the gamma distribution. We analyse how the skewness values depend on $\alpha$ as well as on the distributional parameters. The results are compared with the corresponding values of the quantile skewness $b_2(\alpha)$. Due to property S1, skewness does only depend on the shape parameter of the gamma distribution, but not on the scale parameter. Figure \ref{fig:ThSkewGammaAlpha} depicts skewness as function of $\alpha$ for gamma distributions with different shape parameters.

\begin{figure}
	\centering{
		\includegraphics[scale=0.5]{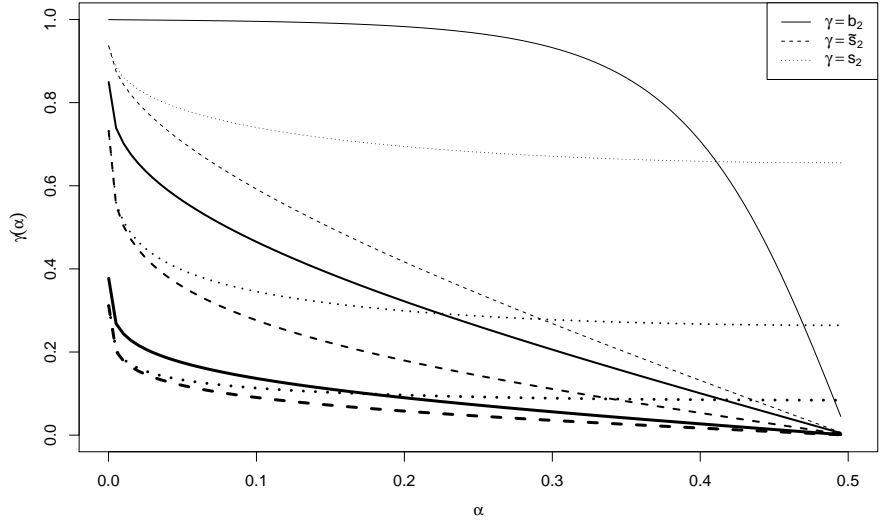}
		\caption{\label{fig:ThSkewGammaAlpha} Expectile and quantile skewness as  function of $\alpha$, for gamma distributions with shape parameters $k = 0.1$, $1$ and $10$, where the line width increases with the value of $k$.}
	}
\end{figure}

First, we only look at the measures $b_2$ and $\tilde{s}_2$ without the correction term $(1 - 2\alpha)^{-1}$. Both of them tend to $0$ as $\alpha$ tends to $1/2$. All curves are strictly decrasing in $\alpha$.
While the quantile skewness exceeds the diagonal for highly skewed distributions, the expectile skewness is restricted to values below the diagonal, corresponding to Theorem \ref{thm:exlSkewStd}. The curves above the diagonal are concave while the ones underneath are convex.

If the expectile skewness is normalized to $1$, it no longer tends to $0$ as $\alpha$ tends to $1/2$. Instead, the still convex curves flatten out with increasing $\alpha$ after a steep decline close to $0$, illustrating that $s_2'(\alpha)$ converges to zero as $\alpha$ tends to $1/2$.
Since the curves flatten out rather quickly, this implies that the limiting expectile skewness $s_3$ is representative of $s_2(\alpha)$ for a considerable part of the range of $\alpha$.

If the quantile skewness is multiplied with the factor $(1 - 2\alpha)^{-1}$, plots show that it also flattens out toward some limiting value with diminishing gradient as $\alpha$ approaches $1/2$. However, these values are then no longer normalized and can be equal to any real number.\par
The observed behaviour is very similar for other popular classes of skewed distributions like the log-normal and the Weibull distribution. If the underlying distribution is skewed to the left, all considered skewness measures increase in $\alpha$ with the expectile skewness curves being concave as long as they stay above the corresponding lower diagonal.

Now we look at the behaviour of the skewness measures $b_2$ and $s_2$ as functions of the shape parameter of the underlying distribution. For the shape parameter $k$ of the gamma distribution specifically, this is depicted in Figure \ref{fig:ThSkewGammaShape}.

\begin{figure}
\centering{
\includegraphics[scale=0.5]{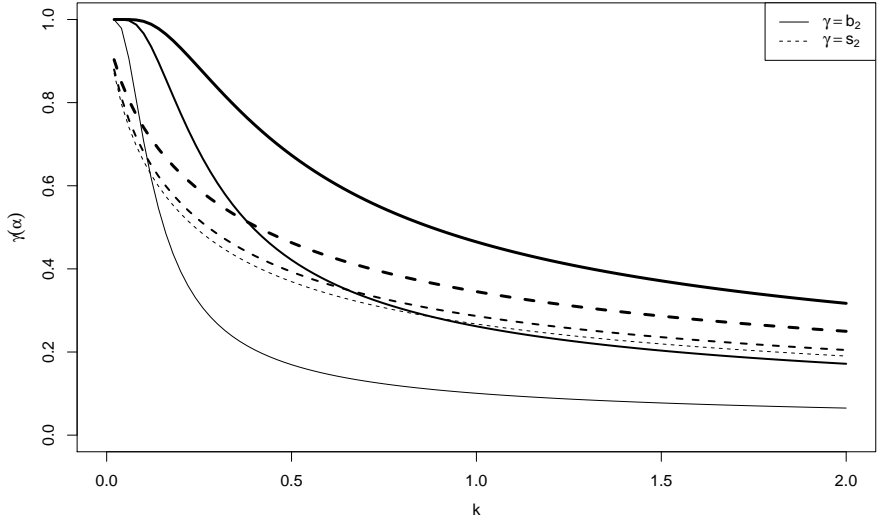}
\caption{\label{fig:ThSkewGammaShape} Expectile and quantile skewness as functions of shape parameter $k$ of the underlying gamma distribution for values $\alpha=0.1,0.25,0.4$, where the line width increase with the value of $\alpha$.}
}
\end{figure}

Both skewness measures decrease in $k$ for all values of $\alpha$. This was to be expected since \citet[pp.60-62]{zwet} showed that $F_{k_1}^{-1} \circ F_{k_2}$ is convex for $k_1 \leq k_2$, where $F_k$ denotes the cdf of the gamma distribution with shape parameter $k$. 
The qualitative behaviour is analogous for similarly ordered classes of distributions like the Weibull distribution, thus strengthening our conjecture that $s_2$ satisfies skewness property S3. The curves of the expectile and quantile skewness differ slightly with the former being strictly convex while the latter become concave for $k$ close to $0$. However, except for very small values of $k$, $b_2$ decreases more rapidly than $s_2$, especially for small values of $\alpha$. The plot also further confirms that the range of $s_2$ for different values of $\alpha$ is substantially smaller than that of $b_2$.

\subsection{Performance of the empirical skewness measures}

In this section, we examine and compare bias and variance of different empirical skewness measures. Here, we include quantile skewness $b_2(\alpha)$, expectile skewness $s_2(\alpha)$, Tajuddin's measure $s_3$ and the moment skewness $\gamma_M = E((X - \mu)/\sigma)^3$. All corresponding empirical measures are obtained as plug-in estimators. Their behaviour is explored for gamma distributions with varying shape parameter $k \in [0.1, 10]$, for the log-normal distribution with fixed log-mean $0$ and varying log-variance $\tau^2 \in [0.01, 2.25]$ as well as for Student's t-distribution with varying degrees of freedom $k \in \mathbb{N} \setminus \{1, 2\} \cup \{\infty\}$ (where $k=\infty$ corresponds to the standard normal distribution). For any skewness measure $\gamma$ with empirical version $\hat{\gamma}_n$, we define the standardized bias, variance and MSE by
\begin{equation*}
\mathrm{sbias}(\hat{\gamma}_n, \gamma) = \frac{\mathrm{bias}(\hat{\gamma}_n, \gamma)}{\gamma}, \quad
\mathrm{sVar}(\hat{\gamma}_n) = \frac{\mathrm{Var}(\hat{\gamma}_n)}{\gamma^2}, \quad
\mathrm{sMSE}(\hat{\gamma}_n, \gamma) = \frac{\mathrm{MSE}(\hat{\gamma}_n, \gamma)}{\gamma^2},
\end{equation*}
if $\gamma \neq 0$, and as their non-standardized versions otherwise. Since the different skewness measures are scaled differently and might also behave differently within the chosen classes of distributions, this makes their biases, variances and MSE's comparable. For the parameter $\alpha$, the values $0.1$, $0.25$ and $0.4$ have been considered, the sample size $n$ varied from $20$ to $10000$, and each simulation  is based on $10000$ repetitions.

\begin{figure}
	\centering{
		\includegraphics[scale=0.45]{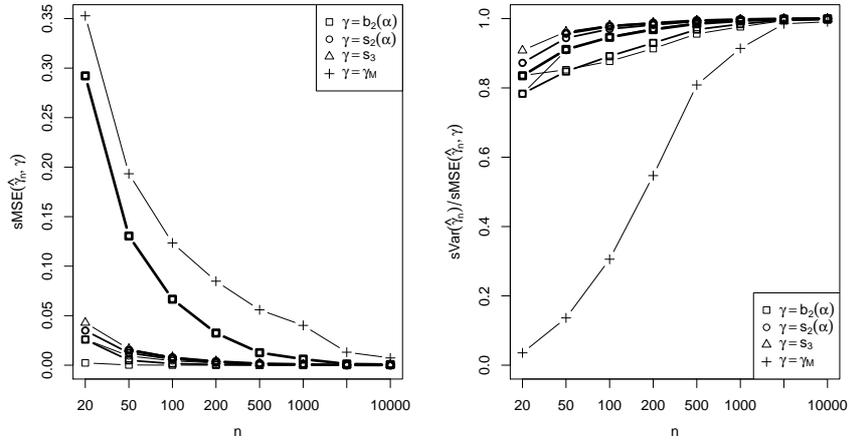}
		\caption{\label{fig:EmpSkewGamma01} Left panel: Standardized MSE's. Right panel: Percentage of the MSE taken up by the variance. Underlying distribution: $\Gamma(0.1, 1)$. Moment skewness: $6.325$. For skewness measures depending on a parameter $\alpha$, increasing line width symbolizes increasing values of $\alpha$.}
	}
\end{figure}

\begin{figure}
	\centering{
		\includegraphics[scale=0.45]{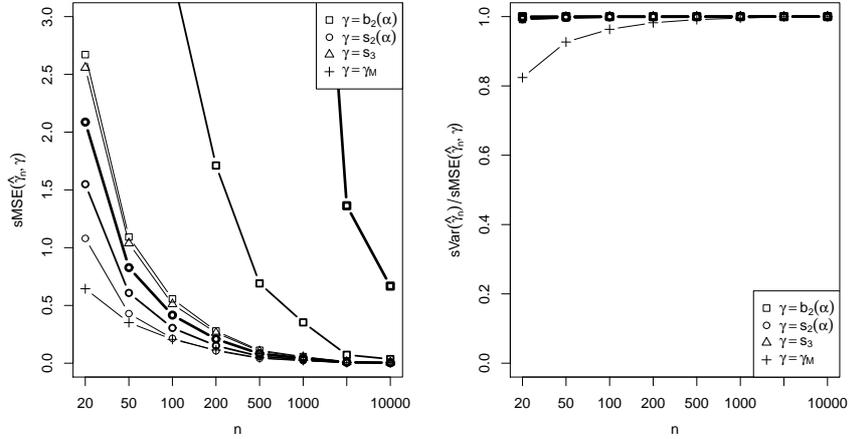}
		\caption{\label{fig:EmpSkewGamma10} Left panel: Standardized MSE's. Right panel: Percentage of the MSE taken up by the variance. Underlying distribution: $\Gamma(10, 1)$. Moment skewness: $0.632$. For skewness measures depending on a parameter $\alpha$, increasing line width symbolizes increasing values of $\alpha$.}
	}
\end{figure}

First, we consider a highly skewed (shape parameter $0.1$, see Figure \ref{fig:EmpSkewGamma01}) and a mildly skewed (shape parameter $10$, see Figure \ref{fig:EmpSkewGamma10}) gamma distribution. We observe that the MSE generally seems to decrease with increasing skewness; however, that decrease is slower for $\gamma_M$ than for the other measures. While $b_2$ and $\gamma_M$ have the highest MSE at either end of the skewness spectrum, the expectile skewness is always in a acceptable range and converges fairly quickly towards $0$ as $n$ increases. While there is almost no bias for the mildly skewed distribution, all measures are at least slightly biased (relative to their variance) for high skewness. For increasing $n$, the bias vanishes. Irrespective of the distributional skewness, $\gamma_M$ is always the most biased measure, for high skewness even to a critical degree.

\begin{figure}
	\centering{
		\includegraphics[scale=0.45]{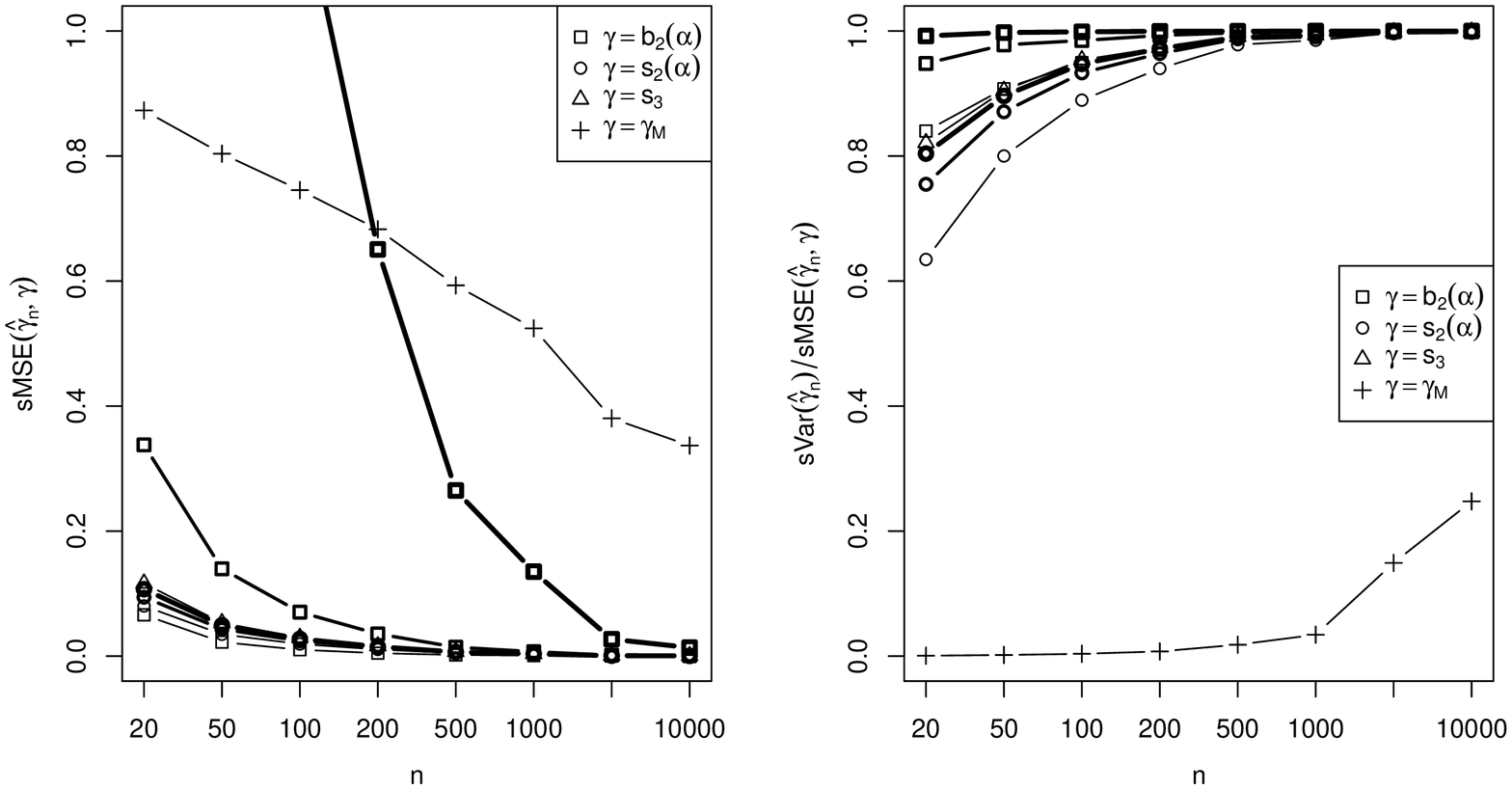}
		\caption{\label{fig:EmpSkewLnorm225} Left panel: Standardized MSE's. Right panel: Percentage of the MSE taken up by the variance. Underlying distribution: $\mathcal{LN}(0, 2.25)$. Moment skewness: $33.468$. For skewness measures depending on a parameter $\alpha$, increasing line width symbolizes increasing values of $\alpha$.}
	}
\end{figure}

\begin{figure}
	\centering{
		\includegraphics[scale=0.45]{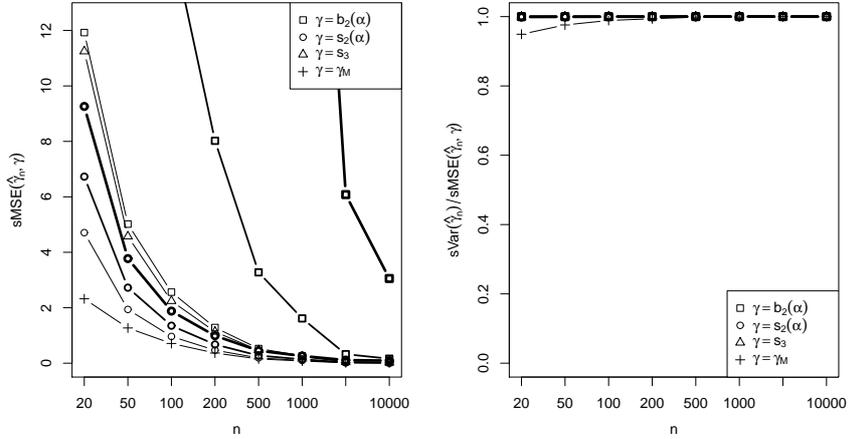}
		\caption{\label{fig:EmpSkewLnorm001} Left panel: Standardized MSE's. Right panel: Percentage of the MSE taken up by the variance. Underlying distribution: $\mathcal{LN}(0, 0.01)$. Moment skewness: $0.302$. For skewness measures depending on a parameter $\alpha$, increasing line width symbolizes increasing values of $\alpha$.}
	}
\end{figure}

The results for the highly skewed (log-variance $2.25$, see Figure \ref{fig:EmpSkewLnorm225}) and the mildly skewed (log-variance $0.01$, see Figure \ref{fig:EmpSkewLnorm001}) log-normal distributions confirm the observations made concerning the gamma distribution. The first log-normal distribution is even more skewed than the first gamma distribution, having the effect that the MSE of $\gamma_M$ is almost completely dominated by the bias. Additionally, the MSE of $\gamma_M$ seems to converge very slowly relative to the other measures, possibly suggesting worse behaviour on heavy-tailed distributions.

\begin{figure}
	\centering{
		\includegraphics[scale=0.45]{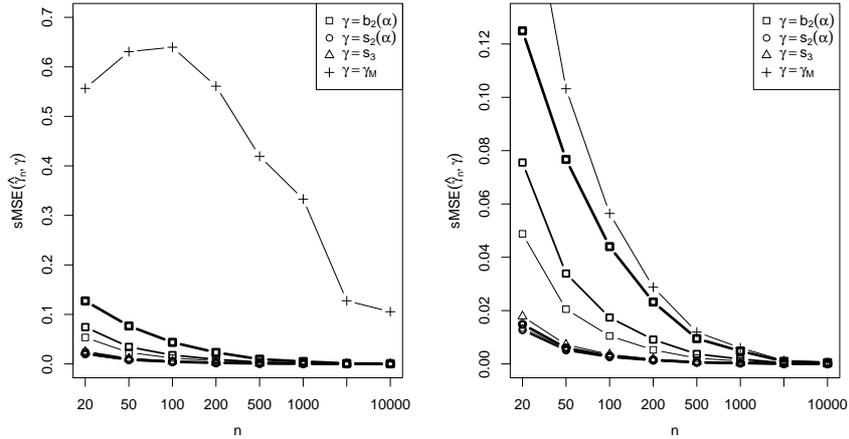}
		\caption{\label{fig:EmpSkewtNorm} Standardized MSE's. Left panel: $t_5$. Right panel: $t_\infty \sim N(0, 1)$. Moment skewness: $0$. For skewness measures depending on a parameter $\alpha$, increasing line width symbolizes increasing values of $\alpha$.}
	}
\end{figure}

Finally, we consider Student's t-distribution and the standard normal distribution (as limiting case) as examples of symmetric distributions. As expected for these distributions, the bias is negligible relative to the variance for all skewness measures. While the MSE basically does not change for $b_2$ and $s_2$ (with the latter achieving even lower values), $\gamma_M$ only behaves nicely for higher degrees of freedom. For lower ones, it becomes somewhat unstable (see Figure \ref{fig:EmpSkewtNorm}), once again showing a poor behaviour for heavy-tailed distributions.

Overall, $s_2(\alpha)$ and $s_3$ seem to be the most stable skewness measures considered here. While they are outperformed for specific distributions, their MSE never explodes, and their bias is always fairly low compared to their variance.

\section{Proofs} \label{sec-proofs}

\begin{proof}[{\slshape Proof of the equivalence of $\le_2$ and (\ref{cto})}]
Convexity of a function $h$ on an interval $I$ is equivalent to
\begin{align} \label{convex-def}
 \frac{h(x)(z-y) + h(y)(x-z) + h(z)(y-x)}{(x-y)(y-z)(z-x)} & \geq 0
\end{align}
for all distinct numbers $x,y,z\in I$ (see, e.g., \citet{artin}).
Equivalently, $h$ is convex, if and only if the numerator of the left hand side of (\ref{convex-def}) is non-negative for all $x<y<z$.

Let $h=G^{-1}\circ F$, where $F$ and $G$  are continuous and strictly increasing on $\{x: 0<F(x)<1\}$ and $\{x: 0<G(x)<1\}$, respectively. Putting $F(x)=u, F(y)=v$ and $F(z)=w$ shows that
\begin{align} \label{convex-def2}
 & G^{-1}(u)(F^{-1}(w)-F^{-1}(v)) + G^{-1}(v)(F^{-1}(u)-F^{-1}(w)) \nonumber \\
 & \hspace{5mm} + G^{-1}(w)(F^{-1}(v)-F^{-1}(u)) \; \geq 0
\end{align}
for all $u<v<w$ is equivalent to the convexity of $G^{-1}\circ F$. Now, a direct computation yields that (\ref{cto}) is equivalent to (\ref{convex-def2}).
\end{proof}

\medskip

\begin{proof}[{\slshape Proof of Proposition \ref{thm:exlSkewStd}}]
Let $X$ be a random variable with cdf $F$ and finite mean $\mu$.
The representation
\begin{equation*}
E(X - t)_+ = \int_{t}^{\infty} (1-F(z)) dz
\end{equation*}
shows that $t \mapsto E(X - t)_+$ is strictly decreasing on $\{t \in \mathbb{R}: F(t) < 1\}$. Similarly, $t \mapsto E(X - t)_-$ is strictly increasing on $\{t \in \mathbb{R}: F(t) > 0\}$.
The mean $\mu = e_X(1/2)$ always lies within the closures of both of these sets.
Since, by Proposition \ref{exp_properties}c), $e_X(\alpha) < \mu < e_X(1 - \alpha)$, we obtain
\begin{align}
			E(X - e_X(1 - \alpha))_+ &< E(X - e_X(\alpha))_+,	\label{eqn:stopLossMon}\\
			E(X - e_X(\alpha))_- &< E(X - e_X(1 - \alpha))_-.	\label{eqn:negStopLossMon}
		\end{align}
		Furthermore, we can rewrite the first order condition for expectiles (\ref{exp_def_2}) for any $\tau \in (0, 1) \setminus \{1/2\}$ in the following two ways
		\begin{align}
			E(X - e_X(\tau))_+ &= \frac{1 - \tau}{1 - 2\tau} (\mu - e_X(\tau)),		\label{eqn:stopLossRewr}\\
			E(X - e_X(\tau))_- &= \frac{\tau}{1 - 2\tau} (\mu - e_X(\tau)).		\label{eqn:negStopLossRewr}
		\end{align}
		Plugging equations (\ref{eqn:stopLossRewr}) and (\ref{eqn:negStopLossRewr}) into inequalities (\ref{eqn:stopLossMon}) and (\ref{eqn:negStopLossMon}) yields
		\begin{align*}
			\alpha (e_X(1 - \alpha) - \mu) &< (1 - \alpha) (\mu - e_X(\alpha)),\\
			\alpha (\mu - e_X(\alpha)) &< (1 - \alpha) (e_X(1 - \alpha) - \mu).
		\end{align*}
Transforming the first inequality yields the upper bound for $\tilde{s}_2(\alpha)$, by transforming the second one we obtain the lower bound.

\medskip		
Let now $X \sim \mathrm{Bin}(1, p)$ for some $p \in (0, 1)$. Some calculations yield
\begin{equation*}
	\tilde{s}_2(\alpha) = (2\alpha - 1) (2 p-1).
\end{equation*}
Hence, $\tilde{s}_2(\alpha) \to 1 - 2\alpha$ for $p \rightarrow 0$ as well as $\tilde{s}_2(\alpha) \to 2\alpha - 1$ for $p \rightarrow 1$. Hence, both inequalities are sharp.
\end{proof}

\medskip

\begin{proof}[{\slshape Proof of Theorem \ref{theorem-A4}}]
From
\begin{equation*}
    EX = \int_0^{\infty} \left( 1-F(x)\right) \, dx - \int_{-\infty}^{0} F(x)\, dx
\end{equation*}
and
\begin{equation*}
    E|X| = \int_0^{\infty} \left( 1-F(x) - F(-x) \right) \, dx
\end{equation*}
we obtain
\begin{eqnarray}
    EX-EY &=& \int_{-\infty}^{\infty} (G-F)(x) \, dx, \label{mean-dif}\\
    E|X|-E|Y| &=& \int_{0}^{\infty} (G-F)(x) \, dx  - \int_{-\infty}^{0} (G-F)(x) \, dx. \label{amean-dif}
\end{eqnarray}
Define $\tilde{X}=(X-\mu_F)/\delta_F$ and $\tilde{Y}=(Y-\mu_G)/\delta_G$ with pertaining cdf's
$\tilde{F}(\cdot)=F(\delta_F \cdot+\mu_F)$ and $\tilde{G}(\cdot)=G(\delta_G \cdot+\mu_G)$. Set $M_{\tilde{G},\tilde{F}}=\tilde{G}(x)-\tilde{F}(x)$.
Since $\mu_{\tilde{F}}=\mu_{\tilde{G}}=0$ and $\delta_{\tilde{F}}=\delta_{\tilde{G}}=1$, (\ref{mean-dif}) and (\ref{amean-dif}) take the form
\begin{eqnarray}
    \int_{-\infty}^{\infty} M_{\tilde{G},\tilde{F}}(x) \, dx &=&0,  \label{M1} \\
    \int_{0}^{\infty} M_{\tilde{G},\tilde{F}}(x) \, dx  &=& \int_{-\infty}^{0} M_{\tilde{G},\tilde{F}}(x)  \, dx.  \label{M2}
\end{eqnarray}

\begin{enumerate}
\item[a)]
Assume $F\leq_2 G$.
Hence, $\tilde{F}$ and $\tilde{G}$ cross each other at most twice.
Since $\mu_{\tilde F}= \mu_{\tilde G}$,  it follows from (\ref{mean-dif}) (and is well-known) that $\tilde F$ and $\tilde G$ are either identical or cross each other at least once.

Now, assume that $M_{\tilde{G},\tilde{F}}\not\equiv 0$. Then, $\tilde{G}^{-1}\tilde{F}$ is strictly convex, and Jensen's inequality implies $\tilde{G}^{-1}\tilde{F}(0)<0$, resulting in $\tilde{F}(0)<\tilde{G}(0)$ (see \cite{zwet}, p.10).

Assume that $M_{\tilde{G},\tilde{F}}$ has exactly one root $x_1$, where $x_1\leq 0$. Put $x_0=-\infty, x_2=0, x_3=\infty$, and
\begin{equation*}
    A_i = \int_{x_{i-1}}^{x_i} M_{\tilde{G},\tilde{F}}(x) \, dx, \quad i=1,2,3.
\end{equation*}
From (\ref{M1}) and (\ref{M2}), we obtain
\begin{equation*}
    A_1+A_2+A_3=0, \qquad A_1+A_2=A_3.
\end{equation*}
Hence, $A_3=0$, which implies that $\tilde{F}(x)=\tilde{G}(x)$ for $x\ge 0$, a contradiction to $\tilde{F}(0)<\tilde{G}(0)$. Since an analogous reasoning excludes a single root
$x_1>0$, it follows that $\tilde{F}$ and $\tilde{G}$ cross each other exactly twice, with $M_{\tilde{G},\tilde{F}}$ changing sign from negative to positive to negative, and $M_{\tilde{G},\tilde{F}}(0)>0$.

The second implication follows from the definitions.
\item[b)]
 Assume $F\leq_\mu^{\delta} G$. Denote the two roots of $M_{\tilde{G},\tilde{F}}$ by $x_1$  and $x_3$, where $x_1<0<x_3$. Further, put $x_0=-\infty, x_2=0, x_4=\infty$, and
\begin{equation*}
    A_i = \int_{x_{i-1}}^{x_i} M_{\tilde{G},\tilde{F}}(x) \, dx, \quad i=1,\ldots,4.
\end{equation*}
From (\ref{M1}) and (\ref{M2}), we obtain
\begin{equation*}
    A_1+A_2+A_3+A_4=0, \qquad A_1+A_2=A_3+A_4
\end{equation*}
(see Figure \ref{fig-FG}). Hence, $-A_1=A_2$, and $-A_4=A_3$, i.e. the area $|A_1|$ equals $A_2$, and $|A_4|$ equals $A_3$.
Consequently,
\begin{equation*}
 \tilde{S}_Y(t) - \tilde{S}_X(t) = \frac1t \int_{-t}^t M_{\tilde{G},\tilde{F}}(t) \, dt  \geq 0.
\end{equation*}
\end{enumerate}
\end{proof}

\begin{figure}
\centering{
\includegraphics[scale=0.45]{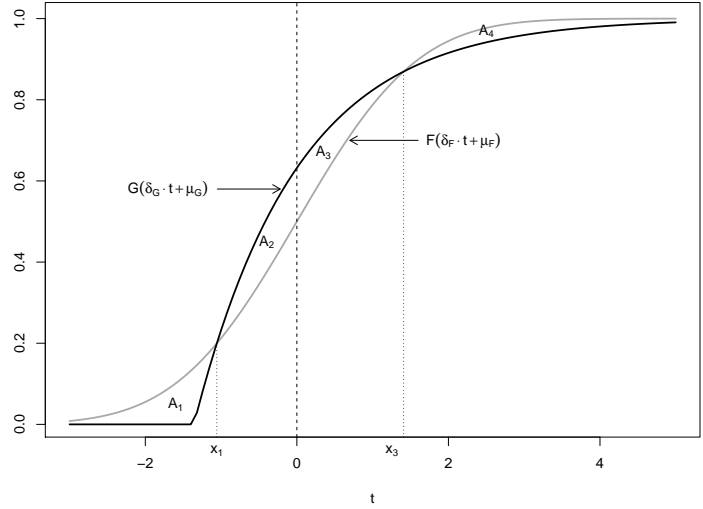}
\caption{\label{fig-FG} Standardized distribution functions with $F\leq_\mu^{\delta} G$.}
}
\end{figure}

\medskip

\begin{proof}[{\slshape Proof of Theorem \ref{asym-skew-function}}]
First, we obtain
\begin{eqnarray}
\lefteqn{ \sqrt{n} \left( S_n(t)-S_X(t)\right) } &&  \nonumber \\
&=& \frac{1}{t\sqrt{n}} \sum_{i=1}^n \left\{ (X_i-\bar{X}-t)_+ - (X_i-\bar{X}+t)_+ - (\pi_X(\mu+t) - \pi_X(\mu-t))  \right\} \nonumber \\
&=& \frac{1}{t\sqrt{n}} \sum_{i=1}^n \left\{ (X_i-\mu-t)_+ - (X_i-\mu+t)_+ - (\pi_X(\mu+t) - \pi_X(\mu-t))  \right\} \nonumber \\
&& + \frac{1}{t\sqrt{n}} \sum_{i=1}^n \left\{ (X_i-\bar{X}-t)_+ - (X_i-\mu-t)_+ \right\} \nonumber \\
&& - \frac{1}{t\sqrt{n}} \sum_{i=1}^n \left\{(X_i-\bar{X}+t)_+ - (X_i-\mu+t)_+ \right\} \nonumber \\
&=& T_{n1} + T_{n2} - T_{n3}, \label{Zerlegung1}
\end{eqnarray}
say. Here, $T_{n2}$ can be written as
\begin{eqnarray*}
T_{n2} &=& \frac{1}{t\sqrt{n}} \sum_{i=1}^n \left\{
(\mu-\bar{X}) \ \mathbbm{1} \{X_i\geq \max(\bar{X}+t,\mu+t)\} \right\} \\
&& + \frac{1}{t\sqrt{n}} \sum_{i=1}^n \left\{
 (X_i-\bar{X}-t) \ \mathbbm{1} \{\bar{X}+t \leq X_i < \mu+t \} \right\} \\
&& - \frac{1}{t\sqrt{n}} \sum_{i=1}^n \left\{
 (X_i-\mu-t) \ \mathbbm{1} \{ \mu+t \leq X_i < \bar{X}+t\} \right\} \\
 &=& T_{n2}^{(1)} + T_{n2}^{(2)} - T_{n2}^{(3)},
\end{eqnarray*}
where $\mathbbm{1}$ denotes the indicator function. Now, we consider  $T_{n2}^{(2)}$.
If $\bar{X} \leq X_i-t \leq \mu$, then $0\leq X_i-t-\bar{X}\leq \mu-\bar{X}$, and, hence,
\begin{eqnarray*}
\left|T_{n2}^{(2)}\right|
&\leq& \sqrt{n}|\bar{X}-\mu| \ \frac{1}{tn} \sum_{i=1}^n
\mathbbm{1} \{\bar{X} \leq X_i-t < \mu \} \\
&=& O_p(1)o_p(1)=o_p(1).
\end{eqnarray*}
Likewise, $T_{n2}^{(3)}=o_p(1)$. Similarly,
\begin{eqnarray*}
\lefteqn{ T_{n2}^{(1)} - \frac{1}{t\sqrt{n}} \sum_{i=1}^n \left\{
(\mu-\bar{X}) \ \mathbbm{1} \{ X_i\geq \mu+t\} \right\} } && \\
&=& \sqrt{n}(\mu-\bar{X}) \ \frac{1}{tn} \sum_{i=1}^n
\mathbbm{1} \{\bar{X} \leq X_i-t \leq \mu \} \ = \ o_p(1).
\end{eqnarray*}
Hence,
\begin{eqnarray*}
T_{n2}^{(1)} &=& \sqrt{n}(\mu-\bar{X}) \, \frac{1}{tn} \sum_{i=1}^n \mathbbm{1} \{X_i\geq \mu+t)\} + o_p(1) \\
&=& \sqrt{n}(\mu-\bar{X}) \, P(X_1 \geq \mu+t)/t + o_p(1).
\end{eqnarray*}
Using $T_{n2}=T_{n2}^{(1)}+o_p(1)$, we get
\begin{eqnarray}
T_{n2} &=& \frac{1}{t\sqrt{n}} \sum_{i=1}^n \left\{ -(X_i-\mu) \, (1-F(\mu+t)) \right\} + o_p(1). \label{Tn2}
\end{eqnarray}
Completely analogous considerations yield
\begin{eqnarray}
T_{n3} &=& \frac{1}{t\sqrt{n}} \sum_{i=1}^n \left\{-(X_i-\mu) \, (1-F(\mu-t)) \right\} + o_p(1). \label{Tn3}
\end{eqnarray}
Plugging (\ref{Tn2}) and (\ref{Tn3}) in (\ref{Zerlegung1}), we finally obtain
\begin{eqnarray*}
\lefteqn{ \sqrt{n} \left( S_n(t)-S_X(t)\right) } && \\
&=& \frac{1}{t\sqrt{n}} \sum_{i=1}^n  \big\{ (X_i-\mu-t)_+ - (X_i-\mu+t)_+
 + (X_i-\mu) (F(\mu+t)-(F(\mu-t)) \\
&& \hspace{16mm} -(\pi_X(\mu+t) - \pi_X(\mu-t))  \big\} + o_p(1),
\end{eqnarray*}
and the central limit theorem yields the assertion.
\end{proof}

\end{document}